\documentclass[draft]{amsart}
\usepackage{dsfont}
\usepackage{amsfonts}
\usepackage{mathrsfs}
\usepackage{amssymb}
\usepackage{amsmath}
\usepackage{amsfonts}
\usepackage{amsfonts,enumerate}
\usepackage{pifont}
\usepackage{enumerate}
\usepackage{graphics}
\usepackage{verbatim}
\usepackage{amsthm}
\usepackage{latexsym,bm}
\usepackage{color}
\numberwithin{equation}{section}
\newtheorem{theorem}{Theorem}[section]

\newtheorem{lemma}[theorem]{Lemma}
\newtheorem{proposition}[theorem]{Proposition}

\theoremstyle{definition}
\newtheorem{remark}[theorem]{Remark}
\newtheorem{definition}[theorem]{Definition}

\newcommand{\M}{{\mathcal M}}

\begin{document}

\title[]{Noncommutative ergodic averages of balls and spheres over Euclidean spaces}

\author{Guixiang Hong}
\address{School of Mathematics and Statistics, Wuhan University, Wuhan 430072  China\\
\emph{E-mail address: guixiang.hong@whu.edu.cn}}

\thanks{\small {{\it MR(2010) Subject Classification}.} Primary 46L52, 37A15; Secondary 46L51, 42B25.}

\maketitle

\begin{abstract}
In this paper, we establish a noncommutative analogue of Calder\'on's transference principle, which allows us to deduce noncommutative ergodic maximal inequalities from the special case---operator-valued maximal inequalities. As applications, we deduce dimension-free estimates of noncommutative Wiener's maximal ergodic inequality and noncommutative Stein-Calder\'on's maximal ergodic inequality over Euclidean spaces. We also show the corresponding individual ergodic theorems. To show Wiener's pointwise ergodic theorem, we construct a dense subset on which pointwise convergence holds following a somewhat standard way. To show Jones' pointwise ergodic theorem, we use again transference principle together with Littlewood-Paley method, which is different from Jones' original variational method that is still unavailable in the noncommutative setting. 
\end{abstract}

\section{Introduction}

Let $\beta_r$ (resp. $\sigma_r$) be the normalized Lebesgue measure on the Euclidean ball $\{v\in\mathbb{R}^n;\;|v|_{\ell_2}\leq r\}$ (resp. the Euclidean sphere $\{v\in\mathbb{R}^n:\;|v|_{\ell_2}=r\}$). Let $(X,m)$ be a standard measure space on which $\mathbb{R}^n$ acts measurably by measure preserving transformation $\pi$. Let $\pi(\beta_r)$ (resp. $\pi(\sigma_r)$) denotes the operator canonically associated to $\beta_r$ (resp. $\sigma_r$) on $L_p(X)$.
Wiener's pointwise ergodic theorem \cite{Wie39} asserts that $\pi(\beta_r)f(x)$ convergence to a limit as $r\rightarrow\infty$ for almost every $x\in X$ provided $f\in L_p(X)$ with $1\leq p<\infty$. The limit is given by $F(f)$, the projection of $f$ to the fixed point subspace $\{g\in L_p(X):\;\pi(v)g=g,\;\forall v\in\mathbb{R}^n\}$. While Jones' pointwise ergodic theorem \cite{Jon93} asserts $\pi(\sigma_r)f(x)$ convergence to $F(f)$ as $r\rightarrow\infty$ for almost every $x\in X$ provided $f\in L_p(X)$ with $n/(n-1)< p<\infty$ and $n>2$ (See \cite{Lac95} for the case $n=2$).

\bigskip

The main tool used in the proof of pointwise convergence is maximal inequality. In both the ball and sphere cases, the ergodic maximal inequalities are deduced, through Calder\'on's transference principle, from the corresponding maximal inequalities in the case $X=\mathbb{R}^n$ and the action given by translation, that is, from the Hardy-Littlewood maximal inequality (in the ball case) and Stein's spherical maximal inequality (in the sphere case). Once the ergodic maximal inequality is available, to show the individual ergodic theorem it suffices to identify a dense subset on which pointwise convergence holds. That the method in constructing a dense subset in the ball case is now standard. However, it is usually a difficult task to identify a dense subset in the sphere case since the spherical measure is singular.  See Jones' original proof \cite{Jon93}.

\bigskip

The main purpose of  this paper is to establish Wiener's and Jones' results in the noncommutative setting. That means, we are going to build maximal ergodic theorems and then pointwise ergodic theorems
for general $W^*$-dynamical system $(\M,\tau,\mathbb{R}^n,\alpha)$, where $(\M,\tau)$ is a von Neumann algebra equipped with a trace $\tau$ and $\alpha:\;\mathbb{R}^n\rightarrow Aut(\M)$ is a continuous trace-preserving group homomorphism (also called an action) in the weak $*$-topology. If we take for $\M$ the algebra $L_\infty(X,m)$ with $(X,m)$ a measure space, $\tau$ the associated integral and $\alpha$ induced by an invertible measure-preserving transformation of $X$, we will recover Wiener's and Jones' results.

\bigskip

For the main purpose, we first establish a noncommutative version of Calder\'on's transference principle. Calder\'on's original arguments do not work in the present setting, since there does not exist perfect notion of ``point'' on the noncommutative measure spaces. We overcome this difficulty using the ideas developed in the theory of vector-valued noncommutative $L_p$ spaces. The noncommutative transference principle reduce maximal ergodic inequalities to the operator-valued maximal inequalities, which have been shown in \cite{Mei07}, \cite{Hon13}. It is worth to mention that in \cite{Mei07}, Mei used noncommutative Doob's inequality \cite{Jun02} to prove the Hardy-Littlewood maximal inequality on $\mathbb R^n$, which yields that the bounds are of order $O(2^n)$. While in \cite{Hon13}, the author show that the bounds could be taken to be independent of $n$ by adapting Stein-Str\"omberg's idea \cite{StSt83}, which is interesting in its own right.

\bigskip

As in the classical setting, with maximal inequality at hand, to use density arguments to show the pointwise convergence, it suffices to find some dense subset such that pointwise convergence holds on it. In Wiener's case, we construct a dense subset following a somewhat standard way, see Section 4 below. However, since sphere measures are singular, the dense subset contructed previously does not work in Jones' case. On the other hand, Jones' original variational method remains an open problem in the noncommutative setting. Our first attempt is via spectral method. We construct dense subsets on which pointwise convergence holds only when the dimension of Euclidean spaces $n\geq4$, see Remark \ref{n34}. Motivated by Rubio de Francia's proof of Stein's spherical maximal inequality (see for instance \cite{Rub86} \cite{Jam09}), we use Littlewood-Paley function to decompose spherical means into pieces, and use transference principle to show each piece satisfies a generalized noncommutative Wiener's ergodic theorem. Summing up all the pieces, we obtain noncommutative Jones' ergodic theorem for all $n\geq3$, see Section 5 below. 

\bigskip

As it is well-known that in the classical setting, Wiener's ergodic theorem inspires many mathematicians to study the ball and sphere averaging problems in ergodic theory associated to more general groups, see for instance the survey paper by Nevo \cite{Nev06}. We expect similar story would take place in the noncommutative setting. Actually, noncommutative ergodic theory has been developed since the very beginning of the theory of ``rings of operators''. However at the early stage, only mean ergodic theorems have been obtained. It is until 1976 after Lance's pioneer work \cite{Lan76} that the study of individual ergodic theorems really took off. Lance proved
that the ergodic averages associated with an automorphism of a
$\sigma$-finite von Neumann algebra which leaves invariant a
normal faithful state converge almost uniformly. Lance's work motivated some mathematicians to study individual ergodic theorems associated to general groups (see for instance \cite{JuXu06} and references therein). All these results can be regarded as individual ergodic theorems in the case $p=\infty$, where maximal ergodic theorems hold trivially. On the other hand, Yeadon
\cite{Yea77} obtained a maximal and pointwise ergodic theorem in the preduals of
semifinite von Neumann algebras. Yeadon's theorem provides a
maximal ergodic inequality which might be understood as a weak
type $(1,1)$ inequality. In contrast
with the classical theory, the noncommutative nature of these
weak type $(1, 1)$ inequalities seems a priori unsuitable for
classical interpolation arguments. The breakthrough was made in the previously quoted paper \cite{JuXu06} by Junge and Xu. They established a sophisticated real interpolation method using well-established noncommutative $L_p$ theory, which together with Yeadon's weak type $(1,1)$ inequality allows them to obtain the noncommutative Dunford-Schwartz maximal ergodic theorem, thus the noncommutative individual ergodic theorem in $L_p$ spaces for all $1\leq p<\infty$. This breakthrough motivates further reserach on noncommutative ergodic theorems including the present paper, see also \cite{Ana06} \cite{Hu08} \cite{Bek08} \cite{Hon} \cite{Lit14} \cite{HoSu} and references therein.

\bigskip

This paper is organized as follows. In the next section, we  give some necessary preliminaries for formulating noncommutative ergodic theorems from the pointview of classical ergodic theory, such as noncommutative analogues of $L_p$ spaces, maximal norms, pointwise convergence and  measure-preserving dynamical systems. In Section 3, we prove a noncommutative version of Calder\'on's transference principle. The noncommutative version of Wiener's ergodic theorem is shown in Section 4. Section 5 is devoted to the proof of noncommutative Jones' ergodic theorem. In the Appendix, we present the spectral method to find a dense subset, which is particularly useful when the underlying group is not amenable.

\section{Preliminaries and framework}

\subsection{Noncommutative $L_p$-spaces}
Let $(\mathcal{M},\tau)$ be a noncommutative measure spaces, that is, $\mathcal M$ is a von Neumann algebra and $\tau$ is a normal
semifinite faithful trace. Let $S^+_{\mathcal{M}}$ be the set
of all positive element $x$ in $\mathcal{M}$ with
$\tau(s(x))<\infty$, where $s(x)$ is the smallest projection $e$
such that $exe=x$. Let $S_{\mathcal{M}}$ be the linear span of
$S^+_{\mathcal{M}}$. Then any $x\in S_{\mathcal{M}}$ has finite
trace, and $S_{\mathcal{M}}$ is a $w^*$-dense $*$-subalgebra of
$\mathcal{M}$.

Let $1\leq p<\infty$. For any $x\in S_{\mathcal{M}}$, the operator
$|x|^p$ belongs to $S^+_{\mathcal{M}}$ ($|x|=(x^*x)^{\frac{1}{2}}$).
We define
$$\|x\|_p=\big(\tau(|x|^p)\big)^{\frac{1}{p}},\qquad\forall x\in S_{\mathcal{M}}.$$
One can check that $\|\cdot\|_p$ is well defined and is a norm on
$S_{\mathcal{M}}$. The completion of $(S_{\mathcal{M}},\|\cdot\|_p)$
is denoted by $L_p(\M)$ which is the usual noncommutative $L_p$-
space associated with $(\M,\tau)$. For convenience, we usually set
$L_{\infty}(\M)=\M$ equipped with the operator norm
$\|\cdot\|_{\M}$. We refer the reader to \cite{PiXu03} for more
information on noncommutative $L_p$-spaces.

\medskip

\subsection{Noncommutative maximal norms}
Let us recall the definition of the noncommutative maximal norm
introduced by Pisier \cite{Pis98} and Junge \cite{Jun02}. We define
$L_p(\M;\ell_\infty)$ to be the space of all sequences
$x=(x_n)_{n\ge1}$ in $L_p(\M)$ which admits a factorization of the
following form: there exist $a, b\in L_{2p}(\M)$ and a bounded
sequence $y=(y_n)$ in $L_\infty(\M)$ such that
 $$x_n=ay_nb, \quad \forall\; n\geq1.$$
The norm of  $x$ in $L_p(\M;\ell_\infty)$ is given by
 $$\|x\|_{L_p(\ell_\infty)}=\inf\big\{\|a\|_{2p}\,
 \sup_{n\geq1}\|y_n\|_\infty\,\|b\|_{2p}\big\} ,$$
where the infimum runs over all factorizations of $x$ as above. We will follow the convention adopted in \cite{JuXu06}  that
$\|x\|_{L_p(\ell_\infty)}$ is sometimes denoted by
 $\big\|\sup_n^+x_n\big\|_p\ .$

More generally, if $\Lambda$ is any index set, we define  $L_p (\M;
\ell_{\infty}(\Lambda))$ as the space of all $x =
(x_{\lambda})_{\lambda \in \Lambda}$ in $L_p (\M)$ that can be
factorized as
 $$x_{\lambda}=ay_{\lambda} b\quad\mbox{with}\quad a, b\in L_{2p}(\M),\; y_{\lambda}\in L_\infty(\M),\; \sup_{\lambda}\|y_{\lambda}\|_\infty<\infty.$$
The norm of $L_p (\M; \ell_{\infty}(\Lambda))$ is defined by
 $$\big \| {\sup_{{\lambda}\in\Lambda}}^{+} x_{\lambda}\big \|_p=\inf_{x_{\lambda}=ay_{\lambda} b}\big\{\|a\|_{2p}\,
 \sup_{{\lambda}\in\Lambda}\|y_{\lambda}\|_\infty\,\|b\|_{2p}\big\} .$$
It is shown in \cite{JuXu06} that $x\in L_p (\M;
\ell_{\infty}(\Lambda))$ if and only if
 $$\sup\big\{\big \| {\sup_{{\lambda}\in J}}^{+} x_{\lambda}\big \|_p\;:\; J\subset\Lambda,\; J\textrm{ finite}\big\}<\infty.$$
In this case, $\big \| {\sup_{{\lambda}\in\Lambda}}^{+}
x_{\lambda}\big \|_p$ is equal to the above supremum. In the following, we omit the index set $\Lambda$ when it will not cause confusion.

\medskip

A closely related operator space is $L_p(\mathcal{M};\ell^c_{\infty})$ for $p\geq2$ which is the set of all sequences $(x_n)_n\subset L_p(\mathcal{M})$ such that
$$\|{\sup_{n\geq1}}^+|x_n|^2\|^{1/2}_{p/2}<\infty.$$
While $L_p(\mathcal{M};\ell^r_{\infty})$ for $p\geq2$ is the Banach space of all sequences $(x_n)_n\subset L_p(\mathcal{M})$ such that $(x^*_n)_n\in L_p(\mathcal{M};\ell^c_{\infty})$. All these spaces fall into the scope of amalgamated $L_p$ spaces intensively studied in \cite{JuPa10}. The following interpolation relationship between symmetric and asymmetric maximal norms will allow the usage of square functions to control maximal norms in Section 5.

\begin{lemma}\label{lem:inter between asy and sy}
Let $2\leq p\leq\infty$. Then
$$(L_p(\mathcal{M};\ell^c_{\infty}),L_p(\mathcal{M};\ell^r_{\infty}))_{\frac12}=L_p(\mathcal{M};\ell_{\infty}).$$
\end{lemma}

\vskip 5 pt

In addition to the strong maximal norms which correspond to $L_p$-norms of maximal function, we are also concerned with the weak maximal norms which correspond to weak $L_p$-norms of maximal function. Given a sequence $(x_n)_n$ in $L_p(\M)$, we define
$$\|(x_n)_n\|_{\Lambda_p(\ell_\infty)}=\sup_{\lambda}\lambda\inf_{e\in\mathcal{P}(\M)}\{(\tau(e^{\perp}))^{\frac 1p}:\;
\|ex_ne\|_\infty\leq\lambda\}.$$
When $\M$ is commutative, then infimum in this definition is attained at the projection $\mathds{1}_{[0,\lambda]}(\sup_n{|x_n|})$. Thus ${\Lambda_p(\ell_\infty)}$-quasi norm is exactly the weak $L_{p}$-quasi norm of maximal function. If we define the quasi Banach space ${\Lambda_p(\M;\ell_\infty)}$ to be the set of all sequence $(x_n)_n\in L_p(\M)$ such that its ${\Lambda_p(\ell_\infty)}$-quasi norm being finite, then these spaces have some nice interpolation properties. We refer the readers to \cite{JuXu06} for more information.

For general index set, ${\Lambda_p(\ell_\infty)}$-quasi norm and whence quasi Banach spaces can be defined similarly as in the definition of strong maximal norms.

\subsection{Noncommutative pointwise convergence}
We recall an appropriate substitute for the usual almost everywhere convergence in the noncommutative setting. This is the almost uniform convergence introduced by Lance \cite{Lan76} (see also \cite{JuXu06}).

Let $(x_{\lambda})_{\lambda \in \Lambda}$ be a family of elements in
$L_p(\mathcal{M}).$ Recall that $(x_{\lambda})_{\lambda\in\Lambda}$ is said to converge almost uniformly to $x,$ abbreviated by $x_{\lambda} \rightarrow x, a.u$ if for every $\varepsilon>0$ there exists a projection $e \in \mathcal{M}$ such that
$$ \tau(1 - e) < \varepsilon \quad \text{and} \quad \lim_{\lambda}\|e ( x_{\lambda} - x )\|_{\infty} = 0.$$
Also, $( x_{\lambda})_{\lambda \in \Lambda}$ is said to converge bilaterally almost uniformly to $x
,$ abbreviated by $x_{\lambda} \rightarrow x, b.a.u$ if for every $\varepsilon>0$ there
is a projection $e \in \M$ such that
$$ \tau( 1 - e ) < \varepsilon \quad \text{and} \quad \lim_{\lambda} \|e ( x_{\lambda} - x )e \|_{\infty} = 0.$$

Obviously, if $x_{\lambda}\rightarrow x, a.u$, then $x_{\lambda}\rightarrow x, b.a.u$. On the other hand, in the commutative case, these two convergences are equivalent to the usual almost everywhere convergence in terms of Egorov's theorem. However they are different in the noncommutative setting.

\medskip

As in \cite{JuXu06}, in order to deduce the pointwise convergence theorems from the
corresponding maximal inequalities, it is convenient to use the closed subspace $L_p(\mathcal{M}; c_0)$ of
$L_p(\mathcal{M};\ell_{\infty}).$ Recall that $L_p(\mathcal{M}; c_0)$ is defined
as the space of all sequences $(x_n) \in L_p(\mathcal{M})$ such that
there are $a, b\in L_{2p}(\mathcal{M})$ and $(y_n)\subset \mathcal{M}$
verifying
$$ x_n = a y_n b\quad \text{and}\quad \lim_ n \|y_n \|_{\infty} = 0.$$
Similarly, for the study of the {\it a.u} convergence, we use the closed subspace $L_p(\mathcal{M}; c^c_0)$ of
$L_p(\mathcal{M};\ell^c_{\infty})$, which is defined to be the space of all sequences $(x_n) \in L_p(\mathcal{M})$ such that
there are $b\in L_{p}(\mathcal{M})$ and $(y_n)\subset \mathcal{M}$
verifying
$$ x_n = y_n b\quad \text{and}\quad \lim_ n \|y_n \|_{\infty} = 0.$$

The following lemma will be useful for our study of individual erogdic theorem (see \cite{DeJu04}).

\begin{lemma}\label{lem:Deju}
\begin{enumerate}[{\rm(i)}]
\item If $1 \leq p <\infty$ and $( x_n ) \in L_p(\mathcal{M}; c_0),$ then $x_n \rightarrow 0, b.a.u.$
\item If $2\leq p < \infty$ and $(x_n )\in L_p (\M; c^c_0),$ then $x_n \rightarrow 0, a.u.$
\end{enumerate}
\end{lemma}

\subsection{Group actions and related notions}

We will call $(\M, \tau, \mathbb R^n,\alpha)$ a $W^*$-dynamical system if $(\M,\tau)$ is a noncommutative measure space and $\alpha:\;\mathbb{R}^n\rightarrow Aut(\M)$ is a continuous trace-preserving group homomorphism (also called an action) in the weak $*$-topology. It is well-known that $\alpha$ is naturally extended to be isometric automorphisms of $L_p(\M)$ for all $1\leq p\leq\infty$, which is still denoted by $\alpha$. As it is well-known, the weak $*$-continuity of $(\alpha(u))_{u\in \mathbb R^n}$ on $\M$ induces the strong continuity of  $(\alpha(u))_{u\in \mathbb R^n}$, i.e. for each $x\in L_p(\M)$ with $1\leq p<\infty$, the map $u\rightarrow \alpha(u)x$ is a continuous map from $\mathbb R^n$ to $L_p(\M)$, where we take the norm topology on $L_p(\M)$.

Let $\mathcal{F}=\{x\in\M:\;\alpha(u)x=x\;,\forall u\in \mathbb R^n\}$. It is easy to show that $\mathcal{F}$ is a von Neumann subalgebra of $\M$ as in the semigroup case \cite{JuXu06}, thus there exists a unique conditional expectation $F:\;\M\rightarrow\mathcal{F}$. Moreover, this conditional expectation extends naturally from $L_p(\M)$ to $L_p(\mathcal{F})$ with $1\leq p<\infty$, which are still denoted by $F$.

Let $M(\mathbb R^n)$ denote the Banach space of bounded complex Borel measures on $\mathbb R^n$. For each $\mu\in M(\mathbb R^n)$, there corresponds an operator $\alpha(\mu)$, with norm bounded by $\|\mu\|_1$ in every $L_p(\M)$, $1\leq p\leq\infty$, given by
$$\alpha(\mu) x=\int_{\mathbb R^n}\alpha(v)xd\mu(v),\;\forall x\in L_p(\M).$$
This definition should be justified as follows. For any $ x\in L_p(\M)$ and $ y\in L_q(\M)$ where $1/p+1/q=1$, the function $v\rightarrow\langle \alpha(v)x,y\rangle$ is continuous on $\mathbb R^n$, bounded by $\|x\|_p\|y\|_q$. Hence
$$\int_{\mathbb R^n}\langle\alpha(v)  x,y\rangle d\mu(v)\leq\|\mu\|_1\|x\|_p\|y\|_q.$$
It follows that the operator $\alpha(\mu) x$ is well defined and $\alpha(\mu) x$ is in $L_p(\M)$. Moreover, $\|\alpha(\mu) x\|_p\leq\|\mu\|_1\|x\|_p$ so that $\|\alpha(\mu)\|\leq\|\mu\|_1$.

It can be easily checked that tha map $\mu\rightarrow \alpha({\mu})$ is a norm-continuous $*$-representation of the involutive Banach algebra $M(\mathbb R^n)$ as an algebra of operators on $L_2(\M)$. We recall that the product in $M(\mathbb R^n)$ is defined as convolution $\mu\ast\nu(f)=\int_{\mathbb R^n}\int_{\mathbb R^n}f(uv)d\mu(u)d\nu(v)$, and the involution is $\mu^*(E)=\overline{\mu(E)^{-1}}$.

Denote by $P(\mathbb R^n)$ the subset of probability measures in $M(\mathbb R^n)$. Let $t\rightarrow \nu_t$ be a weakly continuous map from $\mathbb{R}_+$ to $P(\mathbb R^n)$, namely $t\rightarrow \nu_t(f)$ is continuous for each $f\in C_c(\mathbb R^n)$. We will refer to $(\nu_t)_{t>0}$ as a one-parameter family of probability measures. We can now formulate the following definition.

\begin{definition}\label{def:nc pointwise ergodic family}
A one-parameter family $(\nu_t)_{t>0}\subset P(\mathbb R^n)$ will be called a global (resp. local) noncommutative pointwise ergodic family in $L_p$ if for every $W^*$-dynamical system $(\M,\tau,\mathbb R^n,\alpha)$ and every $x\in L_p(\M)$, $\alpha(\nu_t)x$ converge bilaterally almost uniformly to $F(x)$ (resp. $x$) as $t$ tends to $\infty$ (resp. $0$).
\end{definition}

\section{Noncommutative Calder\'on's transference principle}
In this section, we establish a noncommutative analogue of Calder\'on's transference principle. It is worth to mention that all the assertions in this section are still true for general amenable groups, and thus a general version of noncommutative Calder\'on's transference principle is still available. We prefer to prove them rigorously in our another work, since they are not raleted to the later applications in the present paper. 

Let $\lambda$ be the action of $\mathbb R^n$ on the group itself by translation. Recall that an operator $T$ is completely bounded on $L_p(\mathbb R^n)$ if for any noncommutative measure space $(\mathcal N,tr)$, $T\otimes id_{\mathcal N}$ is bounded on $L_{p}(L_\infty(\mathbb R^n)\overline{\otimes} \mathcal N)$; similar statements hold for a sequence of operators or completely weakly bounded. See \cite{Pis98} for detailed informations.

\begin{theorem}\label{thm:trans principle}
Let $(\nu_t)_{t>0}$ be a family of probability measures on $\mathbb{R}^n$ having their support contained in a Euclidean ball of radius $Ct$ where $C$ is an absolute constant. If the family of operators $(\lambda(\nu_t))_{t>0}$ were completely bounded from $L_p(\mathbb R^n)$ to $L_p(L_\infty(\mathbb R^n);\ell_\infty)$ (resp. to $\Lambda_p(L_\infty(\mathbb R^n);\ell_\infty)$), then for any $W^*$-dynamical system $(\mathcal M,\tau,\mathbb{R}^n,\alpha)$, the family of operators $(\alpha(\nu_t))_{t>0}$ is completely bounded from $L_p(\M)$ to $L_p(\M;\ell_\infty)$ (resp. to $\Lambda_p(\M;\ell_\infty)$).
\end{theorem}

In the rest of this section, we fix a $W^*$-dynamical system $(\mathcal M,\tau,\mathbb{R}^n,\alpha)$. The following three lemmas are well-known in the commutative case, but not obvious in the noncommutative setting.

\begin{lemma}\label{lem:monotinicity of maximal norm}
Let $1\leq p\leq\infty$. Then for $(x_t)_{t>0}$ in $L_p(\M)$, we have
$$\|(x_t)_{t>0}\|_{L_p(\ell_\infty)}=\sup_{T}\|(x_t)_{0<t\leq T}\|_{L_p(\ell_\infty)}$$
and further we have
$$\|(x_t)_{t>0}\|_{\Lambda_p(\ell_\infty)}\simeq\sup_{T}\|(x_t)_{0<t\leq T}\|_{\Lambda_p(\ell_\infty)}$$
for positive $x_t$'s. 
\end{lemma}

\begin{proof}
The first equation has been proved in \cite{JuXu06} using the duality between $L_{p'}(\mathcal M;\ell_1)$ and $L_{p}(\mathcal M;\ell_\infty)$. Let us prove the second equivalence. The left hand side controls trivially the right hand side. Non-trivial part is the reverse inequality.
We start with the observation
\begin{align*}
&\{e\in\mathcal{P}(\M):\|ex_te\|_\infty\leq\lambda,\;\forall t>0\}\\
&\quad\quad=\bigcap_{T}\{e\in\mathcal{P}(\M):\|ex_te\|_\infty\leq\lambda,\;\forall 0<t\leq T\}.
\end{align*}
By density we may assume that $(x_t)_{t>0} \in \ell_\infty(S_\mathcal M)$. Now, given $\lambda>0$ and $T>0$, by definition, there exists a projection $e_T\in \M$ such that 
$$\| e_Tx_t e_T\|_\infty \leq \lambda$$ for any $t \in (0,T]$, and the right hand side dominates modulo a constant 
$$\lambda \tau\big( 1 - e_T \big)^{\frac1p}.$$ 
Define $u=w^*-L_{\infty}-\lim_{\sigma,\mathcal{U}}e_T$. Recall that $u$ is not necessarily a projection. However, recalling that $(x_t)_{t>0} \in \ell_\infty(S_\mathcal M)$, it is straightforward to show that the exact same inequalities above apply for $u$ instead of $e_T$, details are left to the reader. Then, the projection $e= \chi_{[\frac12,1]}(u)$ clearly satisfies $$e \leq 2eue \le 4 u^2 \quad \mbox{and} \quad 1-e \leq 2(1-u).$$ This implies for any $t \in(0,\infty)$ that 
$\big\|ex_te \big\|_\infty \leq 2 \big\| y^*_tu^2y_t \big\|_\infty^\frac12 \leq 2 \lambda$ 
where $y_t$ such that $x_t=y^*_ty_t$,  and $\lambda \tau \big( 1 - e \big)^{\frac1p} \leq  2^{\frac1p}\lambda \tau \big( 1 - u \big)^{\frac1p} $.
Thus we finish the proof.

\end{proof}

\begin{lemma}\label{lem:invariant maximal norms}
Let $1\leq p\leq\infty$. Then for any $v\in\mathbb{R}^n$ and $(x_t)_{t>0}$ in $L_p(\M)$, we have
$$\|(x_t)_{t>0}\|_{L_p(\ell_\infty)}=\|(\alpha(v)x_t)_{t>0}\|_{L_p(\ell_\infty)}$$
and
$$\|(x_t)_{t>0}\|_{\Lambda_p(\ell_\infty)}=\|(\alpha(v)x_t)_{t>0}\|_{\Lambda_p(\ell_\infty)}.$$
\end{lemma}

\begin{proof}
We show the first equality. By the definition of $L_p(\ell_\infty)$-norm, for any $\varepsilon>0$, there exist a factorization  $x_t=ay_tb$ such that 
$$\varepsilon+\|(x_t)_{t>0}\|_{L_p(\ell_\infty)}\geq \|a\|_{2p}\sup_{t>0}\|y_t\|_\infty\|b\|_{2p}.$$
Since $\alpha$ is a homomorphism, we find a factorization of $\alpha(v)x_t$,
$$\alpha(v)x_t=(\alpha(v)a)(\alpha(v)y_t)(\alpha(v)b)$$
with 
$$\|\alpha(v)a\|_{2p}\sup_{t>0}\|\alpha(v)y_t\|_\infty\|\alpha(v)b\|_{2p}=\|a\|_{2p}\sup_{t>0}\|y_t\|_\infty\|b\|_{2p}.$$
Since $\varepsilon$ is arbitrary, we obtain
$$\|(x_t)_{t>0}\|_{L_p(\ell_\infty)}\geq\|(\alpha(v)x_t)_{t>0}\|_{L_p(\ell_\infty)}.$$
The reverse inequality is shown similarly.

\bigskip

The second equality follows from the following observation
\begin{align*}
&\inf_{e\in\mathcal{P}(\M)}\{(\tau(e^{\perp}))^{\frac 1p}:\;
\|e\alpha(v)x_te\|_\infty\leq\lambda\}\\
&=\inf_{e\in\mathcal{P}(\M)}\{(\tau((\alpha(v^{-1})e)^{\perp}))^{\frac 1p}:\;
\|\alpha(v)(\alpha(v^{-1})(e)x_t\alpha(v^{-1})(e))\|_\infty\leq\lambda\}\\
&=\inf_{\alpha(v^{-1})e\in\mathcal{P}(\M)}\{(\tau((\alpha(v^{-1})e)^{\perp}))^{\frac 1p}:\;
\|(\alpha(v^{-1})(e)x_t\alpha(v^{-1})(e)\|_\infty\leq\lambda\}.\\
\end{align*}
where we have used the facts $\alpha(v^{-1})e^\perp=(\alpha(v^{-1})e)^\perp$ and $\alpha$ is a trace-preserving homomorphism.
\end{proof}

\begin{lemma}\label{lem:fubini of maximal norms}
Let $1\leq p\leq\infty$ and $(f_{t})_{t>0}\subset L_p(L_\infty(\mathbb{R}^n)\bar{\otimes}\M)$. Then we have
$$(\int_{\mathbb{R}^n}\|(f_{t}(v))_{t>0}\|^p_{L_p(\M;\ell_{\infty})}dv)^{\frac1p}\leq\|(f_{t})_{t>0}\|_{L_p(L_\infty(\mathbb{R}^n)\bar{\otimes}\M;\ell_{\infty})}
$$
and
\begin{align*}
&\int_{\mathbb R^n}\inf_{e_v\in\mathcal{P}(\M)}\{\tau(e_v^{\perp}):\;
\|e_v f_t(v)e_v\|_\infty\leq\lambda\}dv\\
&\quad\quad\leq \inf_{e\in\mathcal{P}(L_{\infty}(\mathbb{R}^n)\overline{\otimes}\M)}\{\tau\int(e^{\perp}):\;
\|e f_te\|_\infty\leq\lambda\}.
\end{align*}
\end{lemma}

\begin{proof}
By definition, for any $\varepsilon>0$, there exists a factorization $f_{t}=ag_tb$ such that
$$\|a\|_{2p}\sup_t\|g_t\|_\infty\|b\|_{2p}\leq \|(f_{t})_{t>0}\|_{L_p(\ell_{\infty})}+\varepsilon.$$
We see that for each $v\in\mathbb R^n$, we find a factorization $f_{t}(v)=a(v)g_t(v)b(v)$ such that the left hand side is smaller than
\begin{align*}
(\int_{\mathbb{R}^n}\|a(v)\|^p_{2p}\sup_t\|g_t(v)\|^p_\infty\|b(v)\|^p_{2p}dv)^{\frac1p}\leq \|a\|_{2p}\sup_t\|g_t\|_\infty\|b\|_{2p}.
\end{align*}
Since $\varepsilon$ is arbitrary, we obtain the desired result.

\bigskip

Now we show the second inequality. For any $\varepsilon>0$, there exists a projection $e_\varepsilon$ such that
$$\tau\int(e_\varepsilon^{\perp}) \leq\inf_{e\in\mathcal{P}(L_{\infty}(\mathbb{R}^n)\overline{\otimes}\M)}\{\tau\int(e^{\perp}):\;
\|e f_te\|_\infty\leq\lambda\}+\varepsilon.$$
Observing that for each $v\in\mathbb R^n$, we find a projection $e_\varepsilon(v)$ such that 
$$\|e_\varepsilon(v) f_t(v)e_\varepsilon(v)\|_\infty\leq\lambda,$$ so the left hand side is controlled by
$$\int_{\mathbb R^n}\tau(e_\varepsilon(v)^\perp)dv$$
which yields the desired estimate since $\varepsilon$ is arbitrary.

\end{proof}

Now we are at a position to prove Theorem \ref{thm:trans principle}.
\begin{proof}
The case of the strong type $(p,p)$. By Lemma \ref{lem:monotinicity of maximal norm}, it suffices to prove for any $T>0$
\begin{align}\label{desired estimate 1}
\|(\alpha(\nu_t)x)_{0<t\leq T}\|^p_{L_p(\ell_\infty)}\leq C^p_p\|x\|^p_p
\end{align}
for any $x\in L_p(\M)$ where $C_p$ is a constant independent of $x$ and $T$.

Given any $S>0$, by Lemma \ref{lem:invariant maximal norms} and Lemma \ref{lem:fubini of maximal norms}, the left hand side of (\ref{desired estimate 1}) equals
\begin{align*}
&\frac{1}{|B_S|}\int_{B_S}\|(\alpha(v)\alpha(\nu_t)x)_{0<t\leq T}\|^p_{L_p(\M;\ell_\infty)}dv\\
&\quad\quad\leq\frac{1}{|B_S|}\|(\chi_{B_S}(v)\alpha(v)\alpha(\nu_t)x)_{0<t\leq T}\|^p_{L_p(L_\infty(\mathbb{R}^n)\overline{\otimes}\M;\ell_\infty)}.
\end{align*}
Define the $L_p(\M)$-valued function $f$ on $\mathbb{R}^n$ by $f(u)=\chi_{|u|\leq S+CT}(u)\alpha(u)x$. Then for $0<t\leq T$, it is easy to check 
$$\chi_{B_S}(v)\alpha(v)\alpha(\nu_t)x=\chi_{B_S}(v)\lambda\bar{\otimes}id_{\M}(\nu_t)f(v).$$
Thus  the left hand side of (\ref{desired estimate 1}) is not bigger than
\begin{align*}
\frac{1}{|B_S|}\|(\chi_{B_S}(v)\lambda\bar{\otimes}id_{\M}(\nu_t)f(v))_{0<t\leq T}\|^p_{L_p(L_\infty(\mathbb R^n)\overline{\otimes}\M;\ell_\infty)}
\end{align*}
which is smaller than
\begin{align*}
\frac{1}{|B_S|}\|(\lambda\bar{\otimes}id_{\M}(\nu_t)f(v))_{0<t\leq T}\|^p_{L_p(L_\infty(\mathbb R^n)\overline{\otimes}\M;\ell_\infty)}
\end{align*}
by restricting to ${B_S}$ the functions appearing in the factorization of the latter norm. Then by the assumption, we finish the proof
$$\leq\frac{1}{|B_S|}C_p^p\|f\|^p_{p}=\frac{|B_{S+CT}|}{|B_S|}C_p^p\|x\|^p_{p}\rightarrow C_p^p\|x\|^p_{p}$$
as $S\rightarrow\infty$.

\vskip 5 pt

The case of the weak type $(p,p)$: By the fact that every $x\in L_p(\M)$ can be written as linear combination of four positive elements, and Lemma \ref{lem:monotinicity of maximal norm}, it suffices to prove for any $\lambda>0$ and $T>0$
\begin{align}\label{desired estimate 2}
\inf_{e\in\mathcal{P}(\M)}\{\tau(e^{\perp}):\;
\|e\alpha(\nu_t)xe\|_\infty\leq\lambda,\;\forall 0<t\leq T\}\leq C_p^p\frac{\|x\|^p_p}{\lambda^p}.
\end{align}
for any $x\in L^+_p(\M)$ where $C_p$ is a constant independent of $x$ and $T$. Now the arguments in the proof of estimate (\ref{desired estimate 1}) works well for estimate (\ref{desired estimate 2}) using Lemma \ref{lem:invariant maximal norms} and \ref{lem:fubini of maximal norms}.
\end{proof}

\begin{remark}\label{rem:trans for functions}
From the proof, it is easy to observe that if we replace $(\nu_t)_{t>0}$ by a sequence of integrable functions $(\varphi_t)_{t>0}$ with compact supports in Theorem \ref{thm:trans principle}, the same conclusion holds.
\end{remark}

\section{Noncommutative Wiener's ergodic theorem}

In this section, we show a noncommutative version of Wiener's ergodic theorem. The main result is stated as follows.

\begin{theorem}\label{thm:noncommutative wiener ergodic theorem}
The sequence of normalized Lebesgue measures $(\beta_r)_{r>0}$ on the balls in $\mathbb{R}^n$ is a both global and local noncommutative pointwise ergodic family for all $1\leq p<\infty$.
\end{theorem}

As in the classical setting, we use density argument to show the pointwise convergence. We need the following noncommutative Wiener's ergodic maximal inequality. Let $(\M,\tau,\mathbb{R}^n,\alpha)$ be a fixed $W$*-dynamical system.

\begin{theorem}\label{thm:noncommuative wiener ergodic maximal}
Let $x\in L_1(\M)$, then there exists a constant $C_n$ probably depending on $n$ such that 
$$\|(\alpha(\beta_r)x)_{r>0}\|_{\Lambda_{1}(\ell_\infty)}\leq C_n\|x\|_1.$$
That is, for any $\lambda>0$, there exists a projection $e\in\M$ such that
\begin{align}\label{nc wiener ergodic maximal 1}
\|e \alpha(\beta_r)(x)e\|_\infty\leq\lambda,\;\forall r>0\;\mathrm{and}\;\tau(e^{\perp})\leq \frac{C_n\|x\|_1}{\lambda}.
\end{align}

\bigskip

Let $x\in L_p(\M)$ with $1<p\leq\infty$, then there exists a constant $C_{p,n}$ probably depending on $n$ such that
\begin{align}\label{nc wiener ergodic maximal p}
\|(\alpha(\beta_r)x)_{r>0}\|_{L_p(\ell_\infty)}\leq C_{p,n}\|x\|_p.
\end{align}
\end{theorem}

These maximal ergodic inequalities follow from the transference principle---Theorem \ref{thm:trans principle} and the operator-valued Hardy-Littlewood maximal inequalities. Note that the noncommutative version of Hardy-Littlewood maximal inequality was firstly shown by Mei in \cite{Mei07}, his approach is based on noncommutative Doob's inequality, yielding the $O(2^n)$ order of the constants $C_n$ and $C_{p,n}$. It is worth to mention that $C_n$ can be at most of order $O(n)$ and $C_{p,n}$ can be taken to be independent of $n$ if we appeal to the dimension free estimates established in \cite{Hon13}, which are based on Junge/Xu's noncommutative Dunford-Schwartz maximal ergodic theorem \cite{JuXu06} and Stein's original arguments \cite{StSt83}. This type of dimension free estimates in the noncommutative setting should be interesting in its own right that we will not discuss in the present paper.

\bigskip

Let us finish the proof of Theorem \ref{thm:noncommutative wiener ergodic theorem}.

\begin{proof}
(i) The global case. We begin with the definition of a dense subspace. Define $S=\{x-\alpha(v)x:\;x\in L_1(\M)\cap L_\infty(\M),\;v\in\mathbb{R}^n\}$. We claim that $S$ is dense in $\overline{(I-F)L_p(\M)}$ for all $1\leq p<\infty$. It suffices to verify the claim in the case $p=2$, since it is well-known that $L_p(\M)\cap L_2(\M)$ is dense in $L_p(\M)$ for any $1\leq p<\infty$. We only need to prove that $\langle x,z\rangle=0$ $\forall x\in S$ implies $z\in F(L_2(\M))$. Take $x=y-\alpha(v)y$, clearly we have
$$0=\langle x,z\rangle=\langle y,z-\alpha(v)z\rangle$$
which implies $z=F(z)$ since $y\in L_1(\M)\cap L_\infty(\M)$ and $v\in\mathbb{R}^n$ are arbitrary.

Next we show the following fact: For any $x\in S$ and any $1\leq p\leq\infty$, we have
\begin{align}\label{mean ergodic for subspace}
\|\alpha(\beta_r)x\|_p\rightarrow 0,\;\mathrm{as}\;r\rightarrow\infty.
\end{align}
Take $x=y-\alpha(v)y\in S$, then
\begin{align*}
\alpha(\beta_r)x&=\frac{1}{|B_r|}\int_{B_r}\alpha(w)ydw-\frac{1}{|B_r|}\int_{v+B_r}\alpha(w)ydw\\
&=\int_{\mathbb{R}^n}\frac{\chi_{B_r}(w)-\chi_{v+B_r}(w)}{|B_r|}\alpha(w)ydw.
\end{align*}
Hence by Minkowski inequality
\begin{align*}
\|\alpha(\beta_r)x\|_p&\leq \int_{\mathbb{R}^n}\frac{|\chi_{B_r}(w)-\chi_{v+B_r}(w)|}{|B_r|}\|\alpha(w)y\|_pdw\\
&\leq \frac{|B_r\Delta (v+B_r)|}{|B_r|}\|y\|_p\rightarrow 0 \;\mathrm{as}\;r\rightarrow\infty.
\end{align*}
Then (\ref{mean ergodic for subspace}) immediately yields the mean ergodic theorem in $L_p$ for $1\leq p<\infty$, that is, $\|\alpha(\beta_r)x-F(x)\|_p\rightarrow0$ for all $x\in L_p(\M)$.

Now let us show that $(\beta_r)_{r>0}$ is a global noncommutative pointwise ergodic family in $L_1$. That is, for $x\in L_1(\M)$ and fixed $\varepsilon>0$, we want to find a projection $e$ such that $\tau(e^\perp)\leq \varepsilon$ and $\|e(\alpha(\beta_r)x-F(x))e\|_\infty\rightarrow0$, as $r\rightarrow\infty$. First of all, by the fact $S$ is dense in $\overline{(I-F)L_1(\M)}$, for any $\delta>0$ we can find $y\in S$  such that $\|x-F(x)-y\|_1\leq\delta.$ Now apply Theorem \ref{thm:noncommuative wiener ergodic maximal} to the operator $x-F(x)-y$, for any $\lambda>0$, there exists a projection $p$ such that $$\tau(p^\perp)\leq \frac{C\delta}{\lambda},\;\mathrm{and}\;\|p(\alpha(\beta_r)(x-F(x)-y))p\|_\infty\leq\lambda.$$
Take $e=p$ and $\lambda=C\delta/\varepsilon$, we have $\tau(e^{\perp})\leq\varepsilon$ and
\begin{align*}
\|e(\alpha(\beta_r)x-F(x))e\|_\infty&\leq\|e(\alpha(\beta_r)(x-F(x)-y))e\|_\infty+\|e(\alpha(\beta_r)y)e\|_\infty\\
&\leq C\delta/\varepsilon+\|e(\alpha(\beta_r)y)e\|_\infty,
\end{align*}
which tends to $0$ as $r$ tends to $\infty$ due to (\ref{mean ergodic for subspace}) and the fact that $\delta$ can be taken as small as desired.


Fix $1<p<\infty$, by Remark \ref{lem:Deju}, it suffices to prove  $(\alpha(\beta_r)(x)-F(x))_{r>0}\in L_p(\M;c_0)$ for all $x\in L_p(\M)$. By the fact that $L_p(\M;c_0)$ is  Banach space, $S$ is dense in $$\overline{(I-F)L_p(\M)}$$ and the maximal inequality (\ref{nc wiener ergodic maximal p}), we are reduced to prove $(\alpha(\beta_r)x)_{r>0}\in L_p(\M;c_0)$ for all $x\in S$. Fix $x=y-\alpha(v)y\in S$, choose $1<q<p$. Let $0<s<t$, by (\ref{nc wiener ergodic maximal p}), we have
\begin{align*}
\|{\sup_{s<r<t}}^+ \alpha(\beta_r)x\|_p&\leq \|{\sup_{s<r<t}}^+ \alpha(\beta_r)x\|^{q/p}_q(\sup_{s<r<t}\|\alpha(\beta_r)x\|_\infty)^{1-q/p}\\
&\leq (C_q\|x\|_q)^{q/p}(\sup_{s<r<t}\frac{|B_r\Delta (v+B_r)|}{|B_r|}\|y\|_\infty)^{1-q/p},
\end{align*}
which tends to $0$ as $s$ tends to $\infty$.
Thus $(\alpha(\beta_r)x)_{r>0}$ is approximated by $(\alpha(\beta_r)x)_{0<r<s}$'s in $L_p(\M;\ell_\infty)$, whence belongs to $L_p(\M;c_0)$.

\bigskip

(ii) The local case. The dense subspace we consider in this case is the following one
$$\mathcal{D}=\{\alpha(\phi)(x):\;\phi\in C^\infty_c(\mathbb{R}^n),\;x\in L_1(\M)\cap L_\infty(\M)\}.$$
The density is trivial from the fact that the action $\alpha$ is strong continuous in $L_p$ for $1\leq p<\infty$. The fact $(\beta_r)_{r>0}$ is mean ergodic (as $r\rightarrow0$) in $L_p$ for all $1\leq p<\infty$ can also be deduced from the strong continuity of $\alpha$. But we prefer to give a more precise estimate for $x\in\mathcal{D}$. Fix $x=\alpha(\phi)(y)\in\mathcal{D}$, we have
\begin{align}\label{local mean ergodic for subspace}
\|\alpha(\beta_r)(x)-x\|_p&=\|\alpha(\beta_r)\alpha(\phi)(y)-\alpha(\phi)(y)\|_p\\
&\nonumber\leq \|\beta_r\ast\phi-\phi\|_1\|y\|_p\leq C_{\phi}r\|y\|_p,
\end{align}
from which we can also easily deduce the mean ergodicity of $(\beta_r)_{r>0}$. The point we will use is that estimate (\ref{local mean ergodic for subspace}) is also true when $p=\infty$.

Now the fact that $(\beta_r)_{r>0}$ is a local noncommutative pointwise ergodic family for all $1\leq p<\infty$ can be shown using the similar arguments having appeared in the global case, by considering $\mathcal{D}$ (resp. (\ref{local mean ergodic for subspace})) instead of $S$ (resp. (\ref{mean ergodic for subspace})).
\end{proof}

\section{Noncommutative Jones' ergodic theorem}
In this section, we show a noncommutative version of Jones' ergodic theorem. The main result is stated as follows.

\begin{theorem}\label{thm:noncommutative Jones ergodic theorem}
The sequence of normalized Lebesgue measures $(\sigma_r)_{r>0}$ on the spheres in $\mathbb{R}^n$ is a both global and local noncommutative pointwise ergodic family for all $n/(n-1)< p<\infty$, $n\geq3$.
\end{theorem}

As in the proof of Theorem \ref{thm:noncommutative wiener ergodic theorem}, this pointwise ergodic theorem would be established using density argument if we could show the maximal inequality and could find a dense subset on which the pointwise convergence holds. While the maximal inequality now follows easily from the noncommutative transference principle---Theorem \ref{thm:trans principle} and the operator-valued version of Stein's spherical maximal inequality---Proposition 4.1 of \cite{Hon13}, a dense subset is difficult to find since a priori the dense subset $S$ contructed in the proof of Theorem \ref{thm:noncommutative wiener ergodic theorem} is not a good candidate for Jones' ergodic theorem due to the fact the sphere measures are singular.

One way to find a dense subset is via spectral method as done in the commutative setting \cite{NeTh97}. However as we can see in the Appendix that noncommutative analogue of this approach only works for $n\geq4$. Below we present a method which covers the case $n=3$ based on transference principle inspired by the observation in the commutative setting \cite{Jam09}.

We need the following lemma. Fix a $W$*-dynamical system $(\M,\tau,\mathbb{R}^n,\alpha)$. 
\begin{lemma}\label{lem:pointwise for nice function}
Let $p>1$. Let $\varphi$ be a radial smooth compactly supported function on $\mathbb R^n$. For $r>0$, we define $\varphi_r(u)=\frac{1}{r^n}\varphi(\frac{u}{r})$.
Then for any $x\in L_p(\mathcal M)$, $\alpha(\varphi_r)x$ converges to $\int\varphi\cdot F(x)$ (resp. $\int\varphi \cdot x$) b.a.u. as $r\rightarrow\infty$ (resp. $r\rightarrow0$).
\end{lemma}

\begin{proof}
We only show the global case.
By Lemma \ref{lem:Deju}, it suffices to show $(\alpha(\varphi_r)x-\int\varphi\cdot F(x))_r\in L_p(\mathcal M;c_0)$. Since $\varphi$ is radial, we write $\varphi(u)=\varphi_0(|u|)$. Using the fact that $\varphi$ is compactly supported, by polar decomposition and integration by parts, we have
\begin{align*}
\alpha(\varphi_r)x-\int\varphi\cdot F(x)&=-|B(0,1)|\int^\infty_{0}\varphi'_0(s)s^n(\alpha(\beta_{sr})x-F(x))ds,
\end{align*}
where $\varphi'_0$ is the derivative of $\varphi_0$.
Now we clearly obtain the desired result by Wiener's ergodic theorem---Theorem \ref{thm:noncommutative wiener ergodic theorem} and the fact that $L_p(\mathcal M;c_0)$ is a Banach space since
$\int^\infty_{0}|\varphi'_0(s)|s^nds<\infty$.
\end{proof}

Let $\psi_0$ be the Fourier transform of a radial $C^\infty_c(\mathbb R^n)$ function. Assume further that $\psi_0(0) =1$ and that, for $1 < j < n/2$,
$$￼\partial^j_r\psi_0=0$$ where $\partial$ is the radial derivation operator. Such a function can be constructed in the following way: Let $\psi$ be any function that is the Fourier transform of a radial $C^\infty_c(\mathbb R^n)$ function and such that $\psi(0)=1$. For $\xi\in S^{n-1}$ and $r\geq0$, we then define
$$\psi_0(r\xi)= (\sum^n_{j=0}a_jr^{2j}) \psi(r\xi)$$
where the $a_j$'s are chosen inductively so as to have $\psi_0(0) = 1$ and then the required number of derivatives to vanish at 0.

Let us now define for $j \geq 1$, $\psi_j(\xi) = \psi_0(\xi/2^j) - \psi_0({\xi}/{2^{j-1}})$. Note that $\psi_j$ is still the Fourier transform of a radial compactly supported function, and for every $\xi\in\mathbb R^n$,
$$\sum^\infty_{j=0}\psi_j(\xi)=1.$$

For $j\geq0$, let $m_j:=\hat{\sigma}\psi_j$ and $\sigma_{j,1}=\hat{m}_j$. Here $\sigma$ is the normalized Lebesgue measure on the unit sphere $S^{n-1}$
Define $\sigma_{j,r}(u)=\frac{1}{r^n}\sigma_{j,1}(\frac{u}{r})$. Using these notations, it is easy to check that for any $r>0$, we have $\sigma_r=\sum^\infty_{j=0}\sigma_{j,r}$ in the distribution sense. The following estimate plays the key role in showing Theorem \ref{thm:noncommutative Jones ergodic theorem}.

\begin{proposition}\label{pro:key step}
Let $1<p\leq 2$, and $n\geq3$. For each $j\geq0$ and $f\in L_p{(\mathbb R^n;L_p(\mathcal M))}$,
\begin{align}\label{key step}
\|(\sigma_{j,r}\ast f)_{r>0}\|_{L_p(\ell_\infty)}\leq C2^{(n/p-(n-1))j}\|f\|_p.
\end{align}
\end{proposition}

We put off its proof after we show Theorem \ref{thm:noncommutative Jones ergodic theorem}.
\begin{proof}
Interpolating estimate \eqref{key step} with the trivial estimate in $p=\infty$, we have for any $n/(n-1)<p<\infty$, there exists a $Q_p>0$ such that 
\begin{align*}
\|(\sigma_{j,r}\ast f)_{r>0}\|_{L_p(\ell_\infty)}\leq C2^{-Q_pj}\|f\|_p,
\end{align*}
for any $f\in L_p{(\mathbb R^n;L_p(\mathcal M))}$.
Since $\sigma_{j,1}$ has compact support, applying the transference principle---Remark \ref{rem:trans for functions}, we have for
$x\in L_p(\mathcal M)$,
\begin{align*}
\|(\alpha(\sigma_{j,r})x)_{r>0}\|_{L_p(\ell_\infty)}\leq C2^{-Q_pj}\|x\|_p.
\end{align*}
Fix one $x\in L_p(\mathcal M)$.
From this estimate, we see that $\sum^J_{j=0}\alpha(\sigma_{j,r})x$ is uniformly convergent in $L_p(\mathcal M)$, which allows us to identify $\alpha(\sigma_r)x$ as its limit.
On the other hand, from Lemma \ref{lem:pointwise for nice function}, for all $J\in \mathbb N$
$$(\alpha(\sum^J_{j=0}\sigma_{j,r})x-\int \sum^J_{j=0}\sigma_j\cdot F(x))_{r>0}\in L_p(\mathcal M;c_0).$$
Observe that
\begin{align*}
&\|(\alpha(\sigma_r)x-F(x))_{r>0}-(\alpha(\sum^J_{j=0}\sigma_{j,r})x-\int \sum^J_{j=0}\sigma_j\cdot F(x))_{r>0}\|_{L_p(\ell_\infty)}\\
&\leq \|(\alpha(\sigma_r)x-\alpha(\sum^J_{j=0}\sigma_{j,r})x)_{r>0}\|_{L_p(\ell_\infty)}+\|(1-\int \sum^J_{j=0}\sigma_j)\cdot F(x)\|_p\\
&\leq C\sum^\infty_{j=J+1}\|(\alpha(\sigma_{j,r})x)_{r>0}\|_{L_p(\ell_\infty)}+\|(1-\int \sum^J_{j=0}\sigma_j)\cdot F(x)\|_p\\
&\leq C\sum^\infty_{j=J+1}2^{-Q_pj}\|x\|_p+\|(1-\int \sum^J_{j=0}\sigma_j)\cdot F(x)\|_p,
\end{align*}
which tends vers 0, as $J\rightarrow\infty$. Therefore we have $(\alpha(\sigma_r)x-F(x))_{r>0}\in L_p(\mathcal M;c_0)$ due to the fact that $L_p(\mathcal M;c_0)$ is complete. Thus by Lemma \ref{lem:Deju}  we show that 
$(\sigma_r)_{r>0}$ is a global noncommutative pointwise ergodic family for all $n/(n-1)< p<\infty$, $n\geq3$.

\bigskip

The assertion that $(\sigma_r)_{r>0}$ is a local noncommutative pointwise ergodic family for all $n/(n-1)< p<\infty$, $n\geq3$ can be shown similarly, just by replacing $F(x)$ by $x$ in the previous arguments.
\end{proof}

Now let us show Proposition \ref{pro:key step}. We will show estimate \eqref{key step} by establishing the end-point estimates $p=1,2$ and then using noncommutative version of Marcinkiewicz interpolation theorem. It is worth to mention that contrary to the at most one-page proof of classical Marcinkiewicz interpolation theorem, the proof of noncommutative analogue is quite delicate. It was first shown in \cite{JuXu06}, then was improved or further generalized in \cite{BCO17} \cite{Dir15}. Using the noncommuative Marcinkiewicz interpolation theorem, we are reduced to show the end-point estimates.

\begin{lemma}\label{lem:11 estimate}
For all $j\geq0$ and $f\in L_1((\mathbb{R}^n)\overline{\otimes}\M))$, we have
$$\|{\sup_{r}}^+\sigma_{j,r}\ast f\|_{1,\infty}\leq C2^{j}\|f\|_1.$$
\end{lemma}
As in the commutative setting, this estimate follows essentially from the operator-valued Hardy-Littlewood maximal inequality and the fact that  for any $M>n$, there
exists $C_{M}<\infty$ such that
$$|\sigma_j(u)|\leq C_M2^j(1+|u|)^{-M}.$$
It is worth to mention that due to noncommutativity,  elementary inequalities such as $|x+y|\leq|x|+|y|$ do not hold, so we can not just copy the classical arguments. See for instance Theorem 4.3 of \cite{CXY}.

\begin{lemma}
For any $j\geq0$ and $f\in L_2((\mathbb{R}^n)\overline{\otimes}\M))$,
we have
\begin{align*}
\|{\sup_{r>0}}^+\sigma_{j,r}\ast f\|_2\leq
C2^{(1/2-(n-1)/2)j}\|f\|_2.
\end{align*}
\end{lemma}

\begin{proof}
For $j=0$, the method in showing Lemma \ref{lem:11 estimate} applies. We only need to show $j\geq1$.
By density, it suffices to show the desired estimate for $S^+_{\M}$-valued Schwartz function $f$.
We consider the following two column $g$-functions:
$$G_j(f)(u)=\big(\int^{\infty}_0|\sigma_{j,t}\ast f(u)|^2\frac{dt}{t}\big)^{\frac12},$$
and
$$\tilde{G}_j(f)(u)=\big(\int^{\infty}_0|{\tilde{\sigma}_{j,t}\ast f(u)}|^2\frac{dt}{t}\big)^{\frac12},$$
where
$$\tilde{\sigma}_{j,t}=t\frac{d\sigma_{j,t}}{dt}$$
By the fundamental theorem of calculus, we deduce that
\begin{align*}
&({\sigma_{j,t}\ast f})^2=\int^t_{\varepsilon}\frac{d}{ds}(\sigma_{j,s}\ast f)^2ds+(\sigma_{j,\varepsilon}\ast f)^2\\
&=\int^t_{\varepsilon}\tilde{\sigma}_{j,s}\ast f^*\sigma_{j,s}\ast f+\sigma_{j,s}\ast f^*\tilde{\sigma}_{j,s}\ast f\frac{ds}{s}+(\sigma_{j,\varepsilon}\ast f)^2\\
&\leq \int^t_{\varepsilon}|\tilde{\sigma}_{j,s}\ast f^*\sigma_{j,s}\ast f+\sigma_{j,s}\ast f^*\tilde{\sigma}_{j,s}\ast f|\frac{ds}{s}+(\sigma_{j,\varepsilon}\ast f)^2.
\end{align*}
Hence by triangle inequality and H\"older inequality, we have
\begin{align*}
\|{\sup_{t}}^+|{\sigma}_{j,s}\ast f|^2\|^{1/2}_1&\leq\|\int^{\infty}_0|\tilde{\sigma}_{j,s}\ast f^*{\sigma}_{j,s}\ast f+{\sigma}_{j,s}\ast f^*\tilde{\sigma}_{j,s}\ast f|\frac{ds}{s}\|^{1/2}_1\\
&\;\;\;\;\;+\|({\sigma}_{j,\varepsilon}\ast f)^2\|^{1/2}_1\\
&\leq2\|\int^{\infty}_0\tilde{\sigma}_{j,s}\ast f^*{\sigma}_{j,s}\ast f\|^{1/2}_1\\
&\;\;\;\;\;+2\|\int^{\infty}_0{\sigma}_{j,s}\ast f^*\tilde{\sigma}_{j,s}\ast f\frac{ds}{s}\|^{1/2}_1+\|({\sigma}_{j,\varepsilon}\ast f)^2\|^{1/2}_1\\
&\leq4\|G_j(f)\|^{\frac 12}_2\|\tilde{G}_j(f)\|^{\frac 12}_2+\|f_{\varepsilon,j}(x)^2\|^{1/2}_1.\\
&\leq8\|G_j(f)\|^{\frac 12}_2\|\tilde{G}_j(f)\|^{\frac 12}_2.\\
\end{align*}
The last inequality is due to the fact that  $\|f_{\varepsilon,j}(x)^2\|^{1/2}_1$ tends to 0 as $\varepsilon$ tends to $\infty$ by Lebesgue dominated theorem.
The rest of the arguments are similar to those in classical setting \cite{Rub86} or \cite{Jam09}. That is, based on the estimates
$$|\hat{\sigma}(\xi)|+|\nabla\hat{\sigma}(\xi)|\leq
C_n(1+|\xi|)^{(1-n)/2},$$ 
using Plancherel's theorem and the properties of $\psi_j$ to finish the proof. We omit the details.
\end{proof}

\section{Appendix}
In this Appendix, we present the spectral method to find a dense subset in order to prove noncommutative pointwise ergodic theorems, which might be useful in other cases when the transference principle is not available, for instance, when the underlying group is not amenable.

To prove the pointwise ergodic theorem, we first need the following maximal ergodic inequality, which follows from the noncommutative transference principle---Theorem \ref{thm:trans principle} and the operator-valued version of Stein's spherical maximal inequality---Proposition 4.1 of \cite{Hon13}. Let $(\M,\tau,\mathbb{R}^n,\alpha)$ be a fixed $W$*-dynamical system.
\begin{theorem}\label{thm:sphere maximal}
Let $n\geq3$ and $p>n/(n-1)$. Let $x\in L_p(\mathcal M)$. Then there exists a constant $C_{p,n}$
such that
\begin{align}\label{sphere maximal}
\|{\sup_{r>0}}^+{\alpha(\sigma_r) x}\|_p\leq C_{p,n}\|x\|_p.
\end{align}
\end{theorem}

Let $K_n:=SO(n)$ denote the rotation group. The Euclidean motion group $G_n:=\mathbb R^n\rtimes K_n$ with group law of $G_n$ given by
$$(u_1,k_1)(u_2,k_2)=(u_1+k_1u_2, k_1k_2).$$
The inverse of an element $(u,k)\in G_n$ is given by $(-u,k^{-1})$ and $(0,I)$ serves as the identity $e$ where $I$ is the $n\times n$ identity matrix. Then $K_n$ and $\mathbb R^n$ are isomorphic to two subgroups of $G_n$. 

A function $f$ on $\mathbb R^{n}$ is said to be radial if $f$ is a function of $|u|$, equivalently $f(ku)=f(u)$ for all $u\in \mathbb R^n$ and $k\in K_n$. Let $L^1(\mathbb R^n,K_n)$ denote the subspace of radial functions in $L^1(\mathbb R^n)$. This space is canonically identical with $L^1(G_n,K_n)$, the subspace of bi-$K_n$-invariant functions in $L^1(G_n)$. The radial functions on $\mathbb R^n$ form a commutative convolution algebra, since this algebra is canonically isomorphic to the algebra of bi-$K_n$-invariant functions on $G_n$ and it is well-known (see e.g. \cite{FaHa87}, \cite{Str91}) that $(G_n,K_n)$ form a Gelfand pair.

As it is well-known (see e.g. \cite{BJR92}), the complex homomorphisms of a Gelfand pair $L^1(G_n,K_n)$ are given by bounded spherical functions. Bounded $K_n$-spherical functions on $G_n$ are characterized by satisfying $\phi(e)=1$ and the integrals equation
$$\int_{K_n}\phi(akb)dk=\phi(a)\phi(b),\;a,b\in G_n.$$
The family of spherical functions is given for each $s>0$ by
$$\varphi_s(u)=\frac{2^{\frac{n-2}{2}}\Gamma(\frac n2)}{(s|u|)^{\frac{n-2}{2}}}J_{\frac{n-2}{2}}(s|u|)$$
where $J_{\frac {n-2}{2}}$ is the Bessel function of order $\frac{n-2}{2}$, and for $s=0$, the spherical function $\varphi_0(u)=1$ identically. Thus the Gelfand spectrum $\Sigma$ of the algebra $L^1(G_n,K_n)$, equivalently the algebra $L^1(\mathbb R^n,K_n)$, is the union of the Bessel spectrum $(0,\infty)$ and the trivial character $\{0\}$. In what follows $\varphi_s(r)$ stands for $\varphi_s(u)$ with $|u|=r$ when $s>0$ and $\varphi_s(r)=1$ identically when $s=0$.

Let $M(\mathbb R^n,K_n)$ denote the norm-closed convolution algebra generated by the surface measures $\{\sigma_{r}\}_{r>0}$ in $M(\mathbb R^n)$. Since finite linear combinations of functions of the form $\sigma_{r_1}\ast\sigma_{r_2}\ast\dotsm\ast\sigma_{r_k}$ where $r_i>0$ and $k\geq3$ are dense in $L^1(\mathbb R^n,K_n)$ whenever $n\geq3$, and thus $L^1(\mathbb R^n,K_n)\subset M(\mathbb R^n,K_n)$ as a subalgebra. Let $\psi$ be a non-zero continuous complex homomorphism of $M(\mathbb R^n,K_n)$. Then the radial function $\varphi(g)=\psi(m_{K_n}\ast\delta_g\ast m_{K_n})$ equals $\varphi_s(g)$ for some $s\in\Sigma$. Whence, restriction of complex characters from $M(\mathbb R^n,K_n)$ to its subalgebra $L^1(\mathbb R^n,K_n)$ induces a canonical identification of the Gelfand spectrum of the two algebras.

The previous explanations enable us to do the following arguments. Being a weak $*$-continuous action of $\mathbb R^n$ on a von Neumann algebra $\M$, $\alpha$ induces a strongly continuous unitary representation of $\mathbb R^n$ on $L_2(\M)$.  Then $\alpha$ determines canonically a norm continuous $*$-representation of the algebra $M(\mathbb R^n,K_n)$. Let us denote by $\mathcal{A}_{\alpha}$ the commutative $C^*$-algebra which is the closure of $\alpha(M(\mathbb R^n,K_n))$ in the operator norm. Let $\Sigma_\alpha$ denote the spectrum of $\mathcal{A}_\alpha$ which is by definition the set of all non-zero norm continuous complex homomorphisms of $\mathcal{A}_\alpha$.
Clearly $\Sigma_\alpha$ is a subset of the Gelfand spectrum of $M(\mathbb R^n,K_n)$. Consequently, every symmetric (self-adjoint) measure $\mu$ in $M(\mathbb R^n,K_n)$ is mapped to a self-adjoint operator on $L_2(\M)$, whose spectrum is the set $\{\varphi_s(\mu):\;s\in\Sigma_{\alpha}\}$. As a consequence, we have the following spectral decomposition,
\begin{align}\label{spectral decomposition}
\alpha(\sigma_r)=\int_{s\in\Sigma_\alpha}\varphi_s(r)de_{s}.
\end{align}

Moreover, we need the following asymptotic estimates of spherical functions. 

\begin{lemma}\label{lem:pointwise spectral estimates}
Fix $\varepsilon>0$, and let $\Sigma_\varepsilon=\{s:\;\varepsilon\leq s\leq \varepsilon^{-1}\}$. Then
\begin{align}\label{pointwise spectral estimates}
\sup_{s\in\Sigma_{\varepsilon}}\left|\varphi_s(r)\right|\leq {C_{\varepsilon,n}}(1+r)^{-\frac{n-1}{2}},
\end{align}
where $C_{\varepsilon,n}$ is a positive constant independent of $r$.
\end{lemma}

The estimates follow from the standard expansion of the Bessel functions at infinity. See for instance Appendix B.7 of \cite{Gra08} for detailed informations.

\begin{theorem}\label{thm:indvidual ergodic theorem}
Let $n\geq4$.
Let  $x\in L_p(\M)$. 
The family $\alpha(\sigma_r)x$ converges to $F(x)$, b.a.u for $n/(n-1)<p<\infty$ as $r\rightarrow\infty$.
\end{theorem}

\begin{proof}
Let us first prove the case $p=2$. Note that $L_2(\M)$ is  the orthogonal sum of three closed subspace: $\mathcal{H}_1=$ the space of operators invariant under each $\alpha(\sigma_r)$, $\mathcal{H}_{\Sigma}=$ the space of operators whose spectral measure is supported in the Bessel spectrum, and finally, $\mathcal{H}_0=$ the space of operators in the kernel of each $\alpha(\sigma_r)$. Clearly, in the first and the third subspace we have $\alpha(\sigma_r)x-F(x)=0$. Hence by Lemma \ref{lem:Deju}, it suffices to prove $(\alpha(\sigma_r)x)_{r>1}\in L_2(\M;c_0)$ for $x\in \mathcal{H}_{\Sigma}$. By the maximal ergodic theorem---Theorem \ref{thm:sphere maximal}, it suffices to prove $(\alpha(\sigma_r)x)_{r>1}\in L_2(\M;c_0)$ for $x$ in some dense set of $\mathcal{H}_{\Sigma}$. {Indeed, suppose we have $(\alpha(\sigma_r)y)_{r>1}\in L_2(\M;c_0)$ for all $y$ in a dense set of $\mathcal{H}_{\Sigma}$}. Then for fixed $x\in \mathcal{H}_{\Sigma}$, for any $\delta>0$, there exists a $y$ in the dense subset of $\mathcal{H}_{\Sigma}$, such that $\|x-y\|_2\leq\delta$. Therefore by maximal inequality (\ref{sphere maximal})
\begin{align*}
\|(\alpha(\sigma_r)x)_{r>1}-(\alpha(\sigma_r)y)_{r>1}\|_{L_2(\ell_{\infty})}\leq C\|x-y\|_2\leq C\delta.
\end{align*}
Since $\delta$ is arbitrary, $(\alpha(\sigma_r)x)_{r>1}$ is in the closure of $L_2(\M;c_0)$, thus belongs to $L_2(\M;c_0)$ because $L_2(\M;c_0)$ is closed.

For $\varepsilon>0$, let $\mathcal{H}_{\varepsilon}$ be the subspace of operators $x\in L_2(\M)$ whose spectral measure $\langle de(x),x\rangle$ is supported in $\Sigma_{\varepsilon}$ defined in Lemma \ref{lem:pointwise spectral estimates}. The dense set of $\mathcal{H}_{\Sigma}$ we shall consider is
$\mathcal{H}_{\varepsilon}$. For $x\in \mathcal{H}_{\varepsilon}$, by spectral decomposition (\ref{spectral decomposition}) as well as the spectral estimates (\ref{pointwise spectral estimates}), we have
$$\|\alpha(\sigma_r)x\|_2\leq \sup_{s\in\Sigma_{\varepsilon}}|\varphi_{s}(r)|\|x\|_2\leq C_{\varepsilon,n}r^{-\frac{n-1}{2}}\|x\|_2.$$
Then the fact that  the space $L_2(\M;c_0)$ is complete yields that $(\alpha(\sigma_r)x)_{r>1}\in L_2(\M;c_0)$, since
$$(\alpha(\sigma_r)x)_{r>1}=\int^\infty_1(\alpha(\sigma_r)x)_{r=t}dt$$
and
\begin{align}\label{failure of 3}
\int^\infty_1t^{-n/2+1/2}dt<\infty.
\end{align}

For other cases $p\neq2$, using again the maximal ergodic theorem---Theorem \ref{thm:sphere maximal}, it suffices to prove $(\alpha(\sigma_r)x)_{r>1}\in L_p(\M;c_0)$ for $x\in L_1(\mathcal M)\cap \mathcal M$, which is a dense subset of $L_p(\mathcal M)$. Without loss of generality, we assume $p<2$. Find $q$ and $\theta$ such that $(2n-1)/(2n-2)<q<p$ and $1/p=(1-\theta)/q+\theta/2$. Let $x\in L_1(\mathcal M)\cap \mathcal M$. For any $1<t<s$, by maximal inequality (\ref{sphere maximal}),
\begin{align*}
\|(\alpha(\sigma_r)x)_{t<r<s}\|_{L_p(\ell_{\infty})}&\leq \|(\alpha(\sigma_r)x)_{t<r<s}\|^{1-\theta}_{L_q(\ell_{\infty})}\|(\alpha(\sigma_r)x)_{t<r<s}\|^\theta_{L_2(\ell_{\infty})}\\
&\leq C_q^{1-\theta}\|x\|_q\|(\alpha(\sigma_r)x)_{t<r<s}\|^\theta_{L_2(\ell_{\infty})},
\end{align*}
which tends to 0 when $t$ tends to $\infty$ using the result in the case $p=2$. Hence $(\alpha(\sigma_r)x)_{r>1}$ is approximated by $(\alpha(\sigma_r)x)_{1<r<t}$'s, therefore belongs to $L_p(\M;c_0)$.
\end{proof}

\begin{remark}\label{n34}
The reason why the spectral method only works for $n\geq4$ is due to the fact that estimate \eqref{failure of 3} is not true for $n\leq3$.
\end{remark}

\noindent \textbf{Acknowledgement.} The work is partially supported by the NSF of China-11601396, Funds for Talents of China-413100002 and 1000 Young Talent Researcher Programm of China-429900018-101150.

\end{document}